 \documentclass[12pt]{amsart}
\usepackage{epsfig}
\usepackage{epic,eepic}

\usepackage[round]{natbib}
\usepackage{amscd,amssymb,graphics,lineno}
\usepackage{mathrsfs}
\usepackage{amsmath}
\newtheorem{theorem}{Theorem}[section]
\newtheorem{proposition}[theorem]{Proposition}

\newtheorem{question}[theorem]{Question}

\theoremstyle{definition}
\newtheorem{definition}[theorem]{Definition}

\theoremstyle{remark}
\newtheorem{remark}[theorem]{Remark}
\numberwithin{equation}{section}

\newcommand{\Q}{\mathbb{Q}}

\newcommand{\rest}{\upharpoonright}

\def\norm#1{\left\Vert#1\right\Vert}

\begin{document}

\begin{center}
{\large\bf On topological groups with an approximate\\[1.5mm] fixed point property}

\footnotetext{2010 {\it Mathematics Subject Classification}. Primary 47-03,
47H10, 54H25; Secondary 37B05.} \footnotetext{{\it Correspondence to: Brice R.
Mbombo E-mail: brice@ime.usp.br. }}

\vskip.20in

Cleon S. Barroso$^{1}$, Brice R. Mbombo$^{2}$, Vladimir G. Pestov$^{3,4}$  \\[2mm]

 {\footnotesize
         $^{1}$Departamento de Matem\'atica, Universidade Federal
do Cear\'a, Campus do Pici, Bl. 914, 60455-760, Fortaleza, CE, Brazil

{\it e}-mail: cleonbar@mat.ufc.br.
\\[1mm]

$^{2}$ Departamento de Matem\'atica, Instituto de Matem\'atica e Estat\'istica, Universidade de S\~ao Paulo, Rua do Mat\~ao, 1010,
05508-090,  S\~ao Paulo, SP, Brazil

{\it e}-mail: brice@ime.usp.br
\\[1mm]

$^{3}$ Departamento de Matem\'atica, Universidade Federal de Santa Catarina,
Trindade, Florian\'opolis, SC, 88.040-900, Brazil
\\[1mm]

$^{4}$ Department of Mathematics and Statistics, University of Ottawa, Ottawa, ON, K1N 6N5,
  Canada

{\it e}-mail: vpest283@uottawa.ca.}
\end{center}
\vskip .5cm

{\small
  
\subsection*{Abstract} {\em A topological group $G$ has the Approximate Fixed Point (AFP) property on a bounded
  convex subset $C$ of a locally convex space if every continuous affine action of $G$ on $C$ admits a net $(x_i)$, $x_i\in C$, such that
$x_{i}-gx_{i}\longrightarrow 0$ for all $g\in G$. In this work, we study the
relationship between this property and amenability.}
\\[1mm]
\noindent {\bf  Keywords:} Amenable groups, Approximate fixed point property , F\o lner property, Reiter property.

\subsection*{Resumo} {\em Um grupo topol\'ogico $G$ possui a propriedade da aproxima\c c\~ao de pontos fixos (AFP) sob um subconjunto limitado e convexo, $C$,
  de um espa\c co localmente convexo, se cada a\c c\~ao cont\'\i nua de $G$ sobre $C$  pelas transforma\c c\~oes afins admite uma sequ\^encia generalizada $(x_i)$, $x_i\in C$, tal que
  $x_{i}-gx_{i}\longrightarrow 0$ qual que seja $g\in G$. Neste trabalho estudamos o rela\c c\~ao entre os grupos com esta propriedade e os grupos medi\'aveis.}\\[1mm]
\noindent {\bf Palavras-chave:} Grupos Medi\'aveis, Propriedade da aproximac\~ao de pontos fixos, Propriedade de F\o lner, Propriedade de Reiter. 
\\[1mm]
{\bf Se\c c\~ao da Academia:} Ci\^encias Matem\'aticas.
}



\section{Introduction}

One of the most useful known characterizations of amenability is stated in terms of a fixed point property. A classical theorem of \citep{Day} says that a topological group $G$ is amenable
if and only if every continuous affine action of $G$ on a compact convex subset $C$ of a locally convex space has a fixed point, that is, a point $x\in C$ with $g\cdot x=x$ for all $g\in G$.
This result generalizes earlier theorems of \citep{kakutani} and \citep{Markov} obtained for abelian acting groups.
 
At the same time, an active branch of current research is devoted to the existence of approximate fixed points for
single maps. Basically, given a bounded, closed convex set $C$ and a map $f:C\longrightarrow C$, one wants to
find a sequence $(x_n) \subset C$ such that $x_{n}-f(x_{n})\longrightarrow 0$. A sequence with this property will be called an approximate fixed point sequence.

The main motivation for this topic is purely mathematical and comes from several instances of the
failure of the fixed point property in convex sets that are no longer assumed to be compact, cf. \citep{DOBO, EDELSTEIN, FLORET, KLEE, NOWAK} and references therein. One of the most emblematic results on this matter states that if $C$ is a non-compact, bounded, closed convex subset of a normed space, then there exists a Lipschitz map $f :C\longrightarrow C$ such that $\inf_{x\in C}\|x-f(x)\|> 0$ \citep{LIN}. Notice in this case that there is no approximate fixed point sequence for $f$. Previous results of topological flavour were discovered by many authors, including \citep{KLEE} who has characterized the fixed point property in terms of compactness in the framework of
metrizable locally convex spaces. 

Both results give rise to the natural question whether a given space
without the fixed point property might still have the approximate fixed point property. The first
thoughts on this subject were developed in \citep{Cleon1, Cleon2, Cleon3, Cleon4}  in the context of weak topologies. Another mathematical motivation for the study of the
approximate fixed point property is the following open question.

\begin{question}
        Let $X$ be a Hausdorff locally convex space. Assume $C\subset X$ is a sequentially compact convex set and
$f:C\longrightarrow C$ is a sequentially continuous map. Does $f$ have a fixed point?
\end{question}

So far, the best answers for this question were delivered in \citep{Cleon3, Cleon4}. Let us
just summarize the results.

\begin{theorem}[Theorem $2.1$, Proposition $2.5$(i) in \citep{Cleon3, Cleon4}]
Let $X$ be a topological vector space, $C\subset X$ a nonempty convex set, and let $f:C\longrightarrow \overline{C}$ be a map. 
\begin{enumerate}
\item If $C$ is bounded, $f(C)\subset C$ and $f$ is affine, then $f$ has an approximate fixed point sequence.
\item If $X$ is locally convex and $f$ is sequentially continuous with totally bounded range, then $0\in \overline{\{x-f(x): x\in C\}}$. And indeed, $f$ has a fixed point provided that $X$ is metrizable.

\item If $C$ is bounded and $f$ is $\tau$-to-$\sigma(X,Z)$ sequentially continuous, then  $0\in \overline{\{x-f(x): x\in C\}}^{\sigma(X,Z)}$, where $\tau$ is the original topology of $X$ and $Z$ is a subspace of the topological dual $X^{\star}$. And, moreover, $f$ has a $\sigma(X,Z)$-approximate fixed point sequence provided that $Z$ is separable under the
strong topology induced by $X$.
\end{enumerate}
\end{theorem}

The idea of approximate fixed points is an old one. Apparently the first result on this kind
was exploited in \citep{Scarf}, where a constructive method for computing fixed points of continuous
mappings of a simplex into itself was described. Other important works along these lines can be found in 
 \citep{Haz},  \citep{Had}, \citep{Idzik}, \citep{Park}. 
 Approximate fixed point property has a lot of applications in many interesting problems. In \citep{Kalenda} it is proved that a Banach space has the weak-approximate fixed point property if and only if it does not contain any isomorphic copy of $\ell^{1}$. As another instance, it can be used to study the existence of limiting-weak solutions for differential equations in reflexive Banach spaces \citep{Cleon1}.

In this note, we study the existence of common approximate fixed points
for a set of transformations forming a topological group. Not surprisingly, the approximate fixed point property for an acting group
$G$ is also closely related to amenability of $G$, however the relationship appears to be more complex. 

There is an extremely broad variety of known definitions of amenability of a topological group, which are typically equivalent in the context of locally compact groups yet may diverge beyond this class. One would expect the approximate fixed point property (or rather ``properties,'' for they depend on the class of convex sets allowed) to provide a new definition of amenability for some class of groups, and delineate a new class of topological groups in more general contexts. This indeed turns out to be the case.

We show that a discrete group $G$ is amenable if and only if every continuous
affine action of $G$ on a bounded convex subset $C$ of a locally convex space
(LCS) $X$ admits approximate fixed points. For a locally compact group, a
similar result holds if we consider actions on bounded convex sets $C$ which
are complete in the additive uniformity, while in general we can only prove that
$G$ admits weakly approximate fixed points. This criterion of amenability is no
longer true in the more general case of a Polish group, even if amenability of
Polish groups can be expressed in terms of the approximate fixed point property
on bounded convex subsets of the Hilbert space.

We view our investigation as only the first step in this direction, and so we close the article with a discussion of open problems for further research.

\section{Amenability}

Here is a brief reminder of some facts about amenable topological groups. For a
more detailed treatment, see e.g. \citep{Paterson}. All the topologies considered
here are assumed to be Hausdorff.

Let $G$ be a topological group. The {\em right uniform structure} on $G$ has as
basic entourages of the diagonal the sets of the form $U_V = \{ (g,h) \in G \times
G \mid hg^{-1} \in V \},$ where $V$ is a neighbourhood of the identity $e$ in $G$.
This structure is invariant under right translations. Accordingly, a function $f : G
\longrightarrow \mathbb{R}$ is right uniformly continuous if for all $\varepsilon >
0$, there exists a neighbourhood $V$ of $e$ in $G$ such that $xy^{-1}\in V$
implies  $|f(x)-f(y)|< \varepsilon$ for every $x,y\in G$.  Let $RUCB(G)$ denote the
space of all right uniformly continuous functions equipped with the uniform norm.
The group $G$ acts on $RUCB(G)$ on the left continuously by isometries: for all
$g\in G $ and $f\in RUCB(G),\,\,g \cdot f=\,^{g}f$ where $^{g}f(x)=f(g^{-1}x)$ for all $x\in
G$.

\begin{definition}

\label{defam} A topological group $G$ is {\em amenable} if it admits an invariant
{\em mean} on $RUCB(G)$, that is, a positive linear functional $\mathrm{m}$
with $\mathrm{m}(1)=1$, invariant under the left translations.
\end{definition}
Examples of such groups include finite groups, solvable topological groups
(including nilpotent, in particular abelian topological groups) and compact groups.
Here are some more examples:
\begin{enumerate}
\item  The unitary group $\mathcal{U}(\ell^{2})$, equipped with the strong operator topology \citep{Harpe1}.
\item   The infinite symmetric group $S_\infty$ with its unique Polish topology.
 \item The group $\mathcal{J}(\mathrm{k})$ of all formal power series in a variable $x$ that have the form
 $f(x)=x+\alpha_{1}x^{2}+\alpha_{2}x^{3}+....,\,\,\, \alpha_{n}\in \mathrm{k}$,
 where $\mathrm{k}$ is a commutative unital ring \citep{BB}.
\end{enumerate}
Let us also mention some examples of non-amenable groups:

\begin{enumerate}
\item The free discrete group $\mathbb{F}_2$ of two generators. More generally, every locally compact group
containing  $\mathbb{F}_2$ as a closed subgroup.
\item The unitary group $\mathcal{U}(\ell^{2})$, with the uniform operator topology \citep{Harpe2}.
\item The group $\mathrm{Au}t(X,\mu)$ of all measure-preserving automorphisms of a standard Borel measure space $(X,\mu)$, with the uniform topology, i.e. the topology determined by the metric $d(\tau,\sigma)=\mu\{x\in
 X:\,\tau(x)\neq\sigma(x)\}$  \citep{GP1}.
\end{enumerate}

The following is one of the main criteria of amenability in the locally
compact case.
\begin{theorem}[F\o lner's condition]
Let $G$ be a locally compact group and denote $\lambda$ the left invariant Haar
measure. Then $G$ is amenable if and only if $G$ satisfies the F\o lner condition: for
every compact set $F\subseteq G$ and $\varepsilon >0$, there is a Borel set
$U\subseteq G$ of positive finite Haar measure $\lambda(U)$ such that
$\frac{\lambda(xU\bigtriangleup U)}{\lambda(U)}< \varepsilon$ for each $x\in F$.
\end{theorem}

 Recall that a {\em Polish group} is a topological group whose topology is Polish, i.e., separable and completely metrizable.

 \begin{proposition}[See e.g. \citep{YMP}, Proposition $3.7$]
A Polish group $G$ is amenable if and only if every continuous affine action of
$G$ on a convex, compact and metrizable subset $K$ of a LCS $X$ admits a
fixed point.
\end{proposition}

For a most interesting recent survey about the history of amenable groups, see \citep{Grigor}.

\section{Groups with Approximate Fixed Point Property}

\begin{definition}\label{defAFP}
  Let $C$ be a convex bounded subset of a topological vector space $X$. Say that a
topological group $G$ has the {\em approximate fixed point} (AFP) property on
$C$ if every continuous affine action of $G$ on $C$ admits an approximate fixed
point net, that is, a net $(x_{i})\subseteq C$ such that for every $g \in G,\,\,x_{i}-
gx_{i}\longrightarrow 0$.
\end{definition}
We will analyse the AFP property of various classes of amenable topological
groups.

\subsection{Case of discrete groups}

\begin{theorem}
  The following properties are equivalent for a discrete group $G$:
 \begin{enumerate}
   \item \label{iDiscrete} $G$ is  amenable,
   \item \label{iiDiscrete} $G$ has the AFP property on every convex bounded subset of a locally convex space.
 \end{enumerate}
\end{theorem}

\begin{proof}
(\ref{iDiscrete}) $\Rightarrow$ (\ref{iiDiscrete}). Let $G$ a discrete amenable
group acting by continuous affine maps on a bounded convex subset $C$ of a
locally convex space $X$. Choose a F\o lner's $(\Phi_{i})_{i\in I}$ net, that is, a net
of finite subsets of $G$ such that
\[\frac{\vert g \Phi_{i}\bigtriangleup
\Phi_{i}\vert}{\vert \Phi_{i}\vert}\longrightarrow 0\quad \forall g \in G.\]
Now, let
$\gamma\in G$, fix $x\in C$ and define $x_{i}=\dfrac{1}{|\Phi_{i}|}\underset{g\in
\Phi_{i}}{\overset{}{\sum}}gx$. Since $C$ is convex, $x_i\in C$ for all $i\in I$.
Notice that $|\Phi_{i}\setminus \gamma \Phi_{i}|= |\gamma\Phi_{i}\setminus
\Phi_{i}|= \frac{1}{2}|\gamma\Phi_{i}\triangle\Phi_{i}|$ for all $i\in I$. Therefore we
have
\[\begin{array}{llllllll}x_{i}-\gamma x_{i}&=
\dfrac{1}{|\Phi_{i}|}\left[\underset{g\in \Phi_{i}}{\overset{}{\sum}}gx- \underset{g\in \Phi_{i}}{\overset{}{\sum}}\gamma gx \right]
\\\\
&=\dfrac{1}{|\Phi_{i}|}\left[\underset{g\in \Phi_{i}}{\overset{}{\sum}}gx- \underset{h\in \gamma \Phi_{i}}{\overset{}{\sum}}hx \right] \\\\
&=\dfrac{1}{|\Phi_{i}|}\left[\underset{g\in (\Phi_{i}\setminus \gamma \Phi_{i})}{\overset{}{\sum}}gx- \underset{g\in (\gamma\Phi_{i}\setminus \Phi_{i})}{\overset{}{\sum}}gx \right]\\\\
&=\dfrac{|\gamma\Phi_{i}\triangle\Phi_{i}|}{2|\Phi_{i}|}\left[\dfrac{1}{|\Phi_{i}\setminus \gamma \Phi_{i}|}\underset{g\in (\Phi_{i}\setminus \gamma \Phi_{i})}{\overset{}{\displaystyle\sum}}gx-
\dfrac{1}{|\gamma\Phi_{i}\setminus \Phi_{i}|}\underset{g\in (\gamma\Phi_{i}\setminus \Phi_{i})}{\overset{}{\displaystyle\sum}}gx \right]
\end{array}\]
Thus $x_{i}-\gamma x_{i}\in
\dfrac{|\gamma\Phi_{i}\triangle\Phi_{i}|}{2|\Phi_{i}|}(C-C)$ and hence
$x_{i}-\gamma x_{i}\longrightarrow 0$ since $C$ is bounded.

\vskip.2cm

 (\ref{iiDiscrete}) $\Rightarrow$ (\ref{iDiscrete}). Let $G$ be a discrete group acting continuously and by affine transformations on a nonempty compact and convex set $K$ in a locally convex space $X$.
By hypothesis, there is a net $(x_{i})\subseteq K$ such that $\forall  g \in
G,\,\,x_{i}-gx_{i}\longrightarrow 0$. By compactness of $K$, this net has
accumulation points in $K$. Since $\forall \, g \in G,\,\,x_{i}-gx_{i}\longrightarrow
0$, this insures invariance of accumulation points and shows the existence of a
fixed point in $K$. Therefore $G$ is an amenable group by Day's fixed point
theorem mentioned in the Introduction.
\end{proof}

\subsection{Case of locally compact groups}
Recall from \citep{Bourb} the following notion of integration of functions with range in a locally convex space.\\
Let $F$ be a locally convex vector space on $\mathbb{R}$ or $\mathbb{C}$.
$F^{\prime}$ denotes the dual space of $F$ and  $F^{\prime \ast}$ the algebraic dual of  $F^{\prime}$. We identify as
usual $F$ (seen as a vector space without topology) with a subspace of
$F^{\prime \ast}$ by associating to any $z\in F$ the linear form $F^{\prime}\ni
z^{\prime} \longmapsto  \langle z, z^{\prime} \rangle \in \mathbb{R}$.


Let $T$ be a locally compact space and let $\mu$ a positive measure on $T$. A
map $f: T\longrightarrow F$ is essentially $\mu$-{\em integrable} if for every
element $z^{\prime}\in F^{\prime},\,\, \langle z^{\prime},f \rangle$ is essentially
$\mu$-integrable. If $f: T\longrightarrow F$ is essentially $\mu$-integrable, then
$z^{\prime} \longmapsto \displaystyle{\int_{T}^{}} \langle z^{\prime},f
\rangle~\textrm{d}\mu$ is a linear map on $F^{\prime}$, i.e. an element of
$F^{\prime \ast}$. The {\em integral} of $f$ is the element of $F^{\prime \ast}$
denoted $\displaystyle{\int_{T}^{}}f~\textrm{d}\mu$ and defined by the condition:
$ \langle z^{\prime}, \displaystyle{\int_{T}^{}}f~\textrm{d}\mu
\rangle=\displaystyle{\int_{T}^{}} \langle z^{\prime},f \rangle~\textrm{d}\mu$ for
every $z^{\prime}\in F^{\prime}$.

Note that, in general we don't have $\displaystyle{\int_{T}^{}}f~\textrm{d}\mu \in
F$. But we have the following.

\begin{proposition}[\citep{Bourb}, chap. $3$, Proposition $7$]
Let $T$ be a locally compact space, $E$ a LCS and $f: T\longrightarrow E$ a
function with compact support. If $f(T)$ is contained in a complete (with regard to
the additive uniformity) convex subset of $E$, then
$\displaystyle{\int_{T}^{}}f~\textrm{d}\mu \in E$.
\end{proposition}

\begin{theorem}\label{main2}
The following are equivalent for a locally compact group $G$:
 \begin{enumerate}
   \item  \label{iLC} $G$ is amenable,
   \item   \label{iiLC} $G$ has the AFP property on every complete, convex, and bounded subset of a locally convex space.
 \end{enumerate}
\end{theorem}
\begin{proof}

 (\ref{iLC}) $\Rightarrow$ (\ref{iiLC}).
 Let $G$ be a locally compact amenable group acting continuously by affine maps on a complete, bounded, convex subset $C$ of a locally convex space $X$. Again, select a F\o lner net $(F_{i})_{i\in I}$  of compact subsets of $G$ such that $\frac{\lambda(g F_{i}\bigtriangleup F_{i})}{\lambda(F_{i})}\longrightarrow
 0\quad \forall g \in G$.
 Fix $x\in C$ and let $\eta_{x}: G \ni g
 \longmapsto gx \in C$ be the corresponding orbit map. Define
$x_{i}=\dfrac{1}{\lambda(F_{i})}\displaystyle{\int_{F_{i}}^{}}\eta_{x}(g)~\textrm{d}\lambda(g)$.
By the above, this is an element of $C$; the {\em barycenter} of the push-forward
measure $(\eta_x)_{\ast}(\lambda\vert_{F_i})$ on $X$. We have, just like in the
discrete case:
\[\begin{array}{llllllll}x_{i}-\gamma x_{i}&=
\dfrac{1}{\lambda(F_{i})}\left[\displaystyle{\int_{F_{i}}^{}}\eta_{x}(g)~\textrm{d}\lambda(g)- \gamma\displaystyle{\int_{F_{i}}^{}}\eta_{x}(g)~\textrm{d}\lambda(g) \right]
\\\\
&=\dfrac{1}{\lambda(F_{i})}\left[\displaystyle{\int_{F_{i}}^{}}\eta_{x}(g)~\textrm{d}\lambda(g)- \displaystyle{\int_{F_{i}}^{}}\eta_{x}(\gamma g)~\textrm{d}\lambda(g) \right] \\\\
&=\dfrac{1}{\lambda(F_{i})}\left[\displaystyle{\int_{F_{i}}^{}}[\eta_{x}(g)-\eta_{x}(\gamma g)]~\textrm{d}\lambda(g) \right] \\\\&=\dfrac{1}{\lambda( F_{i})}\left[\displaystyle
{\int_{\gamma F_{i}\bigtriangleup F_{i}}^{}}\eta_{x}(v)~\textrm{d}\lambda(v) \right].
\end{array}\]

Now, let $q$ be any continuous seminorm on $C$. We have:
$$q(x_{i}-\gamma x_{i})\leq \frac{\lambda(\gamma F_{i}\bigtriangleup
F_{i})}{\lambda(F_{i})}K$$ where $K=\underset{v\in G}{\overset{}{\sup}}\, q\circ
\eta_{x}(v)< \infty$ since $C$ is bounded. Thus $x_{i}-\gamma
x_{i}\longrightarrow 0$.

\vskip.2cm

 (\ref{iiLC}) $\Rightarrow$ (\ref{iLC}).
 Same argument as in the case of discrete groups.
\end{proof}

The assumption of completeness of $C$ does not look natural in the context of
approximate fixed points, but we do not know if it can be removed. It depends on
the answer to the following.

\begin{question}
Let $f$ be an affine homeomorphism of a bounded convex subset $C$ of a locally
convex space $X$. Can $f$ be extended to a continuous map (hence, a
homeomorphism) of the closure of $C$ in $X$?\\
\end{question}

Nevertheless, we can prove the following.

\begin{theorem}
Every amenable locally compact group $G$ has a weak approximate fixed point
property  on each bounded convex subset $C$ of a locally convex space $X$.
\end{theorem}

\begin{proof}
In the notation of the proof of Theorem \ref{main2}, let
$\mu_{i}=(\eta_x)_{\ast}(\lambda\vert_{F_i})$ denote the push-forward of the
measure $\lambda_{i}=\lambda \rest F_{i}$ along the orbit map $\eta_x\colon G
\ni g\longmapsto gx \in C$. Let $x_{i}$ be the barycenter of $\mu_{i}$. This time,
$x_{i}$ need not belong to $C$ itself, but will belong to the completion $\hat{C}$,
of $C$ (the closure of $C$ in the locally convex vector space completion $\hat
X$).

For every $g\in G$, denote $z_i^g$ the
barycenter of the measure $g.\mu_{i}=(\eta_x)_{\ast}(g\lambda_i)$. Just like in the proof of Theorem \ref{main2}, for every $g$ we have $x_i-z_i^g\to 0$.

Now select a net $\nu_{j}$ of
measures with finite support on $G$, converging to $\lambda_{i}$ in the vague
topology \citep{Bourb}. Denote $y_{j}$ the barycenter of the push-forward measure $(\eta_x)_{\ast}(\nu_j)$.
Then
$y_{j}\Rightarrow x_{i}$ in the vague topology on the space of finite
measures on the compact space $F_{i}.x$. Clearly, $y_j\in C$, and so $g\cdot y_j$ is well-defined and $g\cdot y_{j}\Rightarrow z^g_{i}$ for every $g\in\Phi$.
It follows
that $gy_{j}-y_{j}$ weakly converges to $0$ for every $g\in G$.
\end{proof}

\begin{remark}
Clearly the weak AFP property implies amenability of $G$ as well.
\end{remark}

\subsection{Case of Polish groups}
The above criteria do not generalize beyond the locally compact case in the
ways one might expect: not every amenable non-locally compact Polish group has the AFP property, even on a bounded convex subset of a Banach space.

\begin{proposition}\label{contreexample}
The infinite symmetric group $S_{\infty}$ equipped with its natural Polish
topology does not have the AFP property on closed convex bounded subsets of
$\ell^1$.
\end{proposition}

If we think of $S_\infty$ as the group of all self-bijections of the natural numbers ${\mathbb {N}}$, then the natural (and only) Polish topology on $S_\infty$ is induced from the embedding of $S_\infty$ into the Tychonoff power ${\mathbb {N}}^{\mathbb {N}}$, where ${\mathbb {N}}$ carried the discrete topology.

We will use the following well-known criterion of amenability for locally
compact groups.

\begin{theorem}[Reiter's condition]
  Let $p$ be any real number with $1\leq p < \infty$. A locally compact
group $G$ is amenable if and only if for any compact set $C\subseteq G$ and
$\varepsilon > 0$, there exists $f\in L^{p}(G)$, $f\geq 0$, $\norm{f}_p=1$, such
that: $\|g \cdot f-f\|<\varepsilon$ for all $g\in C$.
\end{theorem}

\begin{proof}[Proof of Proposition \ref{contreexample}]
Denote $\mathrm{prob}(\mathbb{N})$ the
set of all Borel probability measures on $\mathbb{N}$, in other
words, the set of positive functions $b:\mathbb{N}\longrightarrow [0,1]$ such that
$\underset{n\in \mathbb{N}}{\overset{}{\sum}}b(n)=1$. This is the intersection of
the unit sphere of $\ell^{1}$ with the cone of positive elements, a closed convex bounded subset of $\ell^1$.
The Polish group $S_{\infty}$ acts canonically on $\ell^{1}$ by
permuting the coordinates:
$$S_{\infty}\times \ell^{1}(\mathbb{N})\ni (\sigma, (x_{n})_{n})\longmapsto \sigma \cdot(x_{n})_{n}= (x_{\sigma^{-1}(n)})_{n}\in\ell^{1}(\mathbb{N}).$$
Clearly, $\mathrm{prob}(\mathbb{N})$ is invariant and the
restricted action is affine and continuous. We will show that the action of
$S_{\infty}$ on $\mathrm{prob}(\mathbb{N})$
admits no approximate fixed point sequence.

Assume the contrary.
Make the free group $\mathbb{F}_{2}$ act on itself
by left multiplication and identify $ \mathbb{F}_{2}$ with $\mathbb{N}$. In this
way we embed $\mathbb{F}_{2}$ into $S_{\infty}$ as a closed discrete subgroup.
This means that the action of $\mathbb{F}_{2}$ by left regular representation on
$\mathrm{prob}(\mathbb{N})\cong \mathrm{prob}(\mathbb{F}_{2})$ also has
almost fixed points, and $\ell^{1}(\mathbb{F}_{2})$, with regard to the left regular representation of $\mathbb{F}_{2}$, has almost invariant vectors. But this is the
Reiter's condition ($p=1$) for $\mathbb{F}_{2}$, a contradiction with
non-amenability of this group.
\end{proof}

However, it is still possible to characterize amenability of Polish groups in terms of the AFP property.

\begin{theorem}\label{con3}
The following are equivalent for a Polish group $G$:
\begin{enumerate}
  \item  \label{iReflexive} $G$ is amenable,
  \item \label{iiReflexive} $G$ has the AFP property on every bounded, closed and convex subset of the Hilbert space.
  \end{enumerate}
\end{theorem}

\begin{proof}
 (\ref{iReflexive}) $\Rightarrow$  (\ref{iiReflexive}).
 It is enough to show that a norm-continuous affine action of $G$ on a bounded
closed convex subset $C$ of $\ell^{2}$ is continuous with regard to the weak
topology, because then there will be a fixed point in $C$ by Day's theorem.

Let $x\in C$ and $g\in G$ be any, and let $V$ be a weak neighborhood of $g.x$
in $C$. The weak topology on the weakly compact set $C$ coincides with the $\sigma (\mathrm{span}\, C, \ell^{2})$
topology, hence one can choose $x_{1},
x_{2},...,x_{n}\in C$ and $\varepsilon > 0$ so that $y\in V$ whenever $|\langle
x_{i}, y-gx\rangle| < \varepsilon $ for all $i$.

Denote $K$ the diameter of $C$.
Because the action is norm-continuous,
we can find $U\ni e$ in $G$ so that $\norm{u^{-1}x_{i}-x_i}<\varepsilon/2K$ for all $i$.
The set $Ug$
is a neighborhood of $g$ in $G$.

As a weak neighbourhood of $x$, take the set $W$ formed by all $\zeta\in C$ with $|\langle
g^{-1}x_{i}, \zeta-x\rangle| < \varepsilon/2 $ for all $i$. Equivalently, the condition on $\zeta$ can be stated $|\langle
x_{i}, g\zeta-gx\rangle| < \varepsilon/2 $ for all $i$.

If now  $u\in U$ and $\zeta \in W$, one has

\begin{eqnarray*}
 | \langle x_{i}, (ug)\cdot \zeta -gx\rangle|&=&
  \left\vert\langle u^{-1}x_{i}, g\zeta\rangle - \langle x_{i},
  gx\rangle\right\vert
\\
&=&
\left\vert\langle u^{-1}x_{i}, g\zeta\rangle-\langle x_i,g\zeta\rangle\right\vert +\left\vert \langle x_i,g\zeta\rangle - \langle x_{i},
  gx\rangle\right\vert
\\
&\leq& \norm{u x_i-x_i}\cdot K +\frac{\varepsilon}{2} \\
&=&\varepsilon.
\end{eqnarray*}

This shows that $(ug)\cdot \zeta\in V$, and so the action of $G$ on $C$ is continuous with regard to the weak
topology.

\vskip.2cm

(\ref{iiReflexive}) $\Rightarrow$  (\ref{iReflexive}). Suppose that $G$ acts continuously and by affine transformations on a
compact convex and metrizable subset $Q$ of a LCS $E$.
If $C(Q)$ is equipped
with the usual norm topology, then $G$ acts continuously by affine
transformations on the subspace $A(Q)$ of $C(Q)$ consisting of affine
continuous functions on $Q$. Since $Q$ is a metrizable compact set, the space
$C(Q)$ is separable, so is the space $A(Q)$. Fix a dense countable subgroup $H$ of $G$, and
let $F=\{f_{n}:n\in \mathbb{N}\}$
be a dense subset of the closed unit ball of $A(Q)$ which is $H$-invariant. The
map $T:Q\ni x\longmapsto (\frac{1}{n}f_{n}(x))\in \ell^{2}$  is an affine
homeomorphism of $Q$ onto a convex compact subset of $\ell^{2}$. The
subgroup $H$ acts continuously and by affine transformations on the affine
topological copy $T(Q)$ of $Q$ by, the obvious rule $h\cdot T(x)=T(h\cdot x)$. The action
of $H$ extends by continuity to a continuous affine action of $G$ on $T(Q)$. By
hypothesis, $G$ admits an approximative fixed point sequence in $T(Q)$, and
every accumulation point of this sequence is a fixed point since $Q$ is compact.
Therefore $G$ is amenable.

\end{proof}

\section{Discussion and concluding remarks}

\subsection{Approximately fixed sequences}

As we have already noted, we do not know if every locally compact amenable
group has the AFP property on all convex bounded subsets of locally convex
spaces. Another interesting problem is to determine when does an acting group
possess not merely an approximately fixed net, but an approximate fixed
sequence.

Recall that a topological group $G$ is $\sigma$-compact if it is a union of countably many compact subsets.
 It is easy to see that if an
 amenable locally compact group $G$ is $\sigma$-compact, then it admits an approximate fixed sequence for every continuous action by affine maps on a closed bounded convex set.

\begin{question}
Let $G$ be a metrizable separable group acting continuously and affinely on a
convex bounded subset $C$ of a metrizable and locally convex space. If the
action has an approximate fixed net, does there necessarily exist an
approximate fixed sequence?
\end{question}

This is the case, for example if $G$ is the union of a countable chain of
amenable locally compact (in particular, compact) subgroups, and the convex set $C$ is complete.

Recall in this connection that amenability (and thus Day's fixed point property) is
preserved by passing to the completion of a topological group. At the same time, the AFP property is not preserved by completions. Indeed, the group $S_{\infty}^{fin}$ of all permutations of integers with finite support is amenable as a discrete group, and so, equipped with any group topology, will have the AFP property on every bounded convex subset of a locally convex space. However, its completion with regard to the pointwise topology is the Polish group $S_{\infty}$ which, as we have seen, fails the AFP property on a bounded convex subset of $\ell^1$.

\begin{question}
Does every amenable group whose left and right uniformities are equivalent (a
SIN group) have the AFP property on complete convex sets?
\end{question}

\subsection{Distal actions}
Let $G$ be a topological group acting by homeomorphisms on a compact set
$Q$. The flow $(G,Q) $ is called {\em distal} if whenever
$\lim_{\alpha}s_{\alpha} \cdot x =\lim_{\alpha} s_{\alpha} \cdot y$ for some net $s_{\alpha}$
in $G$, then $x = y$. A particular class of distal flows is given by  {\em
equicontinuous} flows, for which the collection of all maps $x\mapsto g\cdot x$
forms an equicontinuous family on the compact space $Q$. We have the
following fixed point theorem:

\begin{theorem}[\citealt{Hahn}]
If a compact affine flow $(G,Q)$ is distal, then there is a $G$-fixed point.
\end{theorem}

An earlier result by \citep{kakutani} established the same for the class of
equicontinuous flows.

\begin{question}
  Is there any approximate fixed point analogue of the above results for distal or equicontinuous actions by a topological group on a (non-compact) bounded convex set $Q$?
\end{question}

\subsection{Non-affine maps}
Historically, Day's theorem (and before that, the theorem of Markov and
Kakutani) was inspired by the classical Brouwer fixed point theorem \citep{Brou}
and its later more general versions, first for Banach spaces \citep{Schauder} and
later for locally convex linear Hausdorff topological spaces \citep{Tycho}.
(Recently it was extended to topological vector spaces \citep{Cauty}). The
Tychonoff fixed point theorem states the following. Let $C$ be a nonvoid compact
convex subset of a locally convex space and let $f\colon C\longrightarrow C$ be
a continuous map. Then $f$ has a fixed point in $C$. The map $f$ is not
assumed to be affine here.

However, for a common fixed point of more than one function, the
situation is completely different. Papers \citep{Boyce} and
\citep{Huneke} contain independent examples of two commuting maps $f,g:
[0,1] \longrightarrow [0,1]$ without a common fixed point. Hence if a common
fixed point theorem were to hold, there should be further restrictions on the nature of transformations beyond amenability, and for Day's theorem, this restriction is that the transformations are affine.

The Tychonoff fixed point theorem is being extended in the context of approximate fixed points. For instance, here is one recent elegant result.

\begin{theorem}[\citealt{Kalenda}]
  Let $X$ be a Banach space. Then every nonempty, bounded, closed, convex subset
  $C\subseteq X$ has the weak AFP property with regard to each continuous map $f:C\longrightarrow C$ if and only if $X$ does not contains an isomorphic copy of $\ell^{1}$.
\end{theorem}

We do not know if a similar program can be pursued for topological groups.

\begin{question}
Does there exist a non-trivial topological group $G$ which has the approximate
fixed point property with regard to every continuous action (not necessarily affine)
on a bounded, closed convex subset of a locally convex space? of a Banach
space?
\end{question}

\begin{question}
  The same, for the weak AFP property.
\end{question}

Natural candidates are the {\em extremely amenable groups}, see e.g. \cite{Pestov2}. A topological group is extremely amenable if every continuous action of $G$ on a compact space $K$ has a common fixed point. The action does not have to be affine, and $K$ is not supposed to be convex. This is a very strong nonlinear fixed point property.

Some of the most basic examples of extremely amenable Polish groups are:

\begin{enumerate}
     \item  The unitary group $U(\ell^2)$ with the strong operator topology \citep{GrM}.
     \item The group ${\mathrm{Aut}}\,(\Q)$ of order-preserving bijections of the rational
numbers with the topology induced from $S_\infty$ \citep{Pestov1}.
     \item The group ${\mathrm{Aut}}\,(X,\mu)$ of measure preserving transformations of a
standard Lebesgue measure space equipped with the weak topology.
\citep{GP1}.
\end{enumerate}

However, at least the group ${\mathrm{Aut}}\,(\Q)$ does not have the AFP
property with regard to continuous actions on the Hilbert space. To see this, one
can use the same construction as in Proposition \ref{contreexample}, together
with the well-known fact that ${\mathrm{Aut}}\,(\Q)$ contains a closed discrete
copy of $\mathbb{F}_2$.

\section*{Acknowedgements}
Cleon S. Barroso is currently as a visiting researcher scholar at the Texas A$\&$M University. He takes the opportunity to express his gratitude to Prof. Thomas Schlumprecht for his support and friendship. Also, he acknowledges Financial Support form CAPES by the Science Without Bordes Program, PDE 232883/2014-9. Brice R. Mbombo was
supported by FAPESP postdoctoral grant, processo 12/20084-1. Vladimir G. Pestov is a Special Visiting Researcher of the program Science Without Borders of CAPES (Brazil), processo 085/2012.

\bibliographystyle{agsm}
\bibliography{biblio}
\nocite{*}
\end{document}